\documentclass[11pt]{amsart}
\usepackage{amsfonts,amssymb,amsthm,eucal,color,bbm,verbatim,enumitem}

\usepackage[margin=1.25in]{geometry}

\newtheorem{thm}{Theorem}
\newtheorem{cor}[thm]{Corollary}
\newtheorem{lemma}[thm]{Lemma}
\newtheorem{prop}[thm]{Proposition}

\newcommand{\R}{\mathbb{R}}
\newcommand{\E}{\mathbb{E}}

\newcommand{\Prob}{\mathbb{P}}
\newcommand{\N}{\mathbb{N}}

\newcommand{\C}{\mathbb{C}}

\newcommand{\ds}{\displaystyle}

\renewcommand{\Im}{\operatorname{Im}}
\newcommand{\abs}[1]{\left\vert #1 \right\vert}
\newcommand{\norm}[1]{\left\Vert #1 \right\Vert}

\DeclareMathOperator{\var}{Var}

\DeclareMathOperator{\Ent}{Ent}

\newcommand{\Unitary}[1]{\mathbb{U}\left(#1\right)}
\newcommand{\SUnitary}[1]{\mathbb{SU}\left(#1\right)}
\newcommand{\Orthogonal}[1]{\mathbb{O}\left(#1\right)}
\newcommand{\SOrthogonal}[1]{\mathbb{SO}\left(#1\right)}
\newcommand{\Symplectic}[1]{\mathbb{S}\mathbbm{p}\left(2 #1\right)}
\newcommand{\Circle}{\mathbb{S}^1}
\newcommand{\SOneg}[1]{\mathbb{SO}^{-}\left(#1\right)}

\newcommand{\ind}[1]{\mathbbm{1}_{#1}}

\allowdisplaybreaks[3]

\author{Elizabeth S.\ Meckes and Mark W.\ Meckes}

\thanks{E.\ Meckes's research is partially supported by the American Institute
of Mathematics and NSF grant DMS-0852898.}
\thanks{M.\ Meckes's research is
partially supported by NSF grant DMS-0902203.}

\address{Department of Mathematics, Case Western Reserve University,
10900 Euclid Ave., Cleveland, Ohio 44106, U.S.A.}

\email{elizabeth.meckes@case.edu}

\address{Department of Mathematics, Case Western Reserve University,
10900 Euclid Ave., Cleveland, Ohio 44106, U.S.A.}

\email{mark.meckes@case.edu}

\title{Spectral measures of powers of random matrices}

\begin{document}

\begin{abstract}
  This paper considers the empirical spectral measure of a power of a
  random matrix drawn uniformly from one of the compact classical
  matrix groups.  We give sharp bounds on the $L_p$-Wasserstein
  distances between this empirical measure and the uniform measure on
  the circle, which show a smooth transition in behavior when the
  power increases and yield rates on almost sure convergence when the
  dimension grows.  Along the way, we prove the sharp logarithmic
  Sobolev inequality on the unitary group.
\end{abstract}

\maketitle

\section{Introduction}

The eigenvalues of large random matrices drawn uniformly from the
compact classical groups are of interest in a variety of fields,
including statistics, number theory, and mathematical physics; see
e.g.\ \cite{Diaconis} for a survey.  An important general phenomenon
discussed at length in \cite{Diaconis} is that the eigenvalues of an
$N \times N$ random unitary matrix $U$, all of which lie on the circle $\Circle = \{z
\in \C : \abs{z} = 1\}$, are typically more evenly spread out than $N$
independently chosen uniform random points in $\Circle$.  It was found
by Rains \cite{Ra97} that the eigenvalues of $U^N$
are exactly distributed as $N$ independent random points in $\Circle$;
similar results hold for other compact Lie groups.  In subsequent work
\cite{Ra03}, Rains found that in a sense, the eigenvalues of $U^m$
become progressively more independent as $m$ increases from $1$ to
$N$.

In this paper we quantify in a precise way the degree of uniformity of
the eigenvalues of $U^m$ when $U$ is drawn uniformly from any of the
classical compact groups $\Unitary{N}$, $\SUnitary{N}$,
$\Orthogonal{N}$, $\SOrthogonal{N}$, and $\Symplectic{N}$.  We do this
by bounding, for any $p \ge 1$, the mean and tails of the $L_p$-Wasserstein distance $W_p$
between the empirical spectral measure $\mu_{N,m}$ of $U^m$ and the
uniform measure $\nu$ on $\Circle$ (see Section \ref{S:wasserstein}
for the definition of $W_p$).  In particular, we show in Theorem
\ref{T:mean-dist} that
\begin{equation}\label{E:mean}
\E W_p (\mu_{N,m}, \nu) \le C p \frac{\sqrt{m\left[\log
      \left(\frac{N}{m}\right) + 1 \right]}}{N}.
\end{equation}
Theorem \ref{T:dist-concentration} gives a subgaussian tail bound for
$W_p(\mu_{N,m}, \nu)$, which is used in Corollary \ref{T:unitary-BC}
to conclude that if $m = m(N)$, then with probability $1$, for all
sufficiently large $N$,
\begin{equation}\label{E:as-rate}
W_p(\mu_{N, m}, \nu) \le C p \frac{\sqrt{m \log(N)}}{N^{\min\{1, 1/2 + 1/p\}}}.
\end{equation}

In the case $m=1$ and $1\le p\le 2$, \eqref{E:mean} and
\eqref{E:as-rate} are optimal up
to the logarithmic factor, since the $W_p$-distance from
the uniform measure $\nu$ to \emph{any} probability
measure supported on $N$ points is at least $cN^{-1}$.
When $m = N$, Rains's theorem says that
$\mu_{N,N}$ is the empirical measure of $N$ independent uniformly
distributed points in $\Circle$, for which the estimate in
\eqref{E:mean} of order $N^{-1/2}$ is optimal (cf.\ \cite{dBGM}).  We
conjecture that Theorem \ref{T:mean-dist} and Corollary
\ref{T:unitary-BC}, which interpolate naturally between these extreme
cases, are optimal up to logarithmic factors in their entire parameter
space.

In the case that $m=1$ and $p=1$, these results improve the authors'
earlier results in \cite{MM12} (where $W_1(\mu_{N,1}, \nu)$ was
bounded above by $CN^{-2/3}$) to what we conjectured there was the
optimal rate; the results above are completely new for $m > 1$ or $p >
1$.

The proofs of our main results rest on three foundations: the fact
that the eigenvalues of uniform random matrices are determinantal
point processes, Rains's representation from \cite{Ra03} of the
eigenvalues of powers of uniform random matrices, and logarithmic
Sobolev inequalities.  In Section \ref{S:mrep}, we combine some
remarkable properties of determinantal point processes with Rains's
results to show that the number of eigenvalues of $U^m$ contained in
an arc is distributed as a sum of independent Bernoulli random
variables.  In Section \ref{S:meansvars}, we estimate the means and
variances of these sums, again using the connection with determinantal
point processes.  In Section \ref{S:wasserstein}, we first generalize
the method of Dallaporta \cite{Da12} to derive bounds on mean
Wasserstein distances from those data and prove Theorem
\ref{T:mean-dist}. Then by combining Rains's results with tensorizable
measure concentration properties which follow from logarithmic Sobolev
inequalities, we prove Theorem \ref{T:dist-concentration} and
Corollary \ref{T:unitary-BC}.  We give full details only for the case
of $\Unitary{N}$, deferring to Section \ref{S:other_groups} discussion
of the modifications necessary for the other groups.

In order to carry out the approach above, we needed the sharp
logarithmic Sobolev inequality on the full unitary group, rather than 
only on $\SUnitary{N}$ as in \cite{MM12}.  It has been noted
previously (e.g.\ in \cite{HiPeUe,AGZ}) that such a result is clearly
desirable, but that because the Ricci tensor of $\Unitary{N}$ is
degenerate, the method of proof which works for $\SUnitary{N}$,
$\SOrthogonal{N}$, and $\Symplectic{N}$ breaks down.  In the appendix, we prove the logarithmic Sobolev inequality on
$\Unitary{N}$ with a constant of optimal order.

\section{A miraculous representation of the eigenvalue counting
  function}\label{S:mrep}

As discussed in the introduction, a fact about the eigenvalue
distributions of matrices from the compact classical groups which we
use crucially is that they are determinantal point processes.  For
background on determinantal point processes the reader is referred to
\cite{HKPV06}.  The basic definitions will not be repeated here since
all that is needed for our purposes is the combination of Propositions
\ref{T:evs_detpp} and \ref{T:evs_detpp2} with Proposition
\ref{T:detpp_Bernoulli} and Lemma \ref{T:kernel_mean_var} below.  The
connection between eigenvalues of random matrices and determinantal
point processes has been known in the case of the unitary group at least
since \cite{Dyson}.  For the other groups, the earliest reference we
know of is \cite{KaSa}.  Although the language of determinantal point
processes is not used in \cite{KaSa}, Proposition \ref{T:evs_detpp}
below is essentially a summary of \cite[Section 5.2]{KaSa}. We first
need some terminology.

Given an eigenvalue $e^{i\theta}$, $0 \le \theta < 2\pi$, of a unitary
matrix, we refer to $\theta$ as an eigenvalue angle of the
matrix. Each matrix in $\SOrthogonal{2N+1}$ has $1$ as an eigenvalue,
each matrix in $\SOneg{2N+1}$ has $-1$ as an eigenvalue, and each
matrix in $\SOneg{2N+2}$ has both $-1$ and $1$ as eigenvalues; we
refer to all of these as trivial eigenvalues.  Here $\SOneg{N} = \{ U
\in \Orthogonal{N} : \det U = -1\}$, which is considered primarily as
a technical tool in order to prove our main results for
$\Orthogonal{N}$.  The remaining eigenvalues of matrices in
$\SOrthogonal{N}$ or $\Symplectic{N}$ occur in complex conjugate
pairs.  When discussing $\SOrthogonal{N}$, $\SOneg{N}$, or
$\Symplectic{N}$, we refer to the eigenvalue angles corresponding to
the nontrivial eigenvalues in the upper half-circle as nontrivial
eigenvalue angles.  For $\Unitary{N}$, all the eigenvalue angles are
considered nontrivial.

\begin{prop} \label{T:evs_detpp} The nontrivial eigenvalue angles of
  uniformly distributed random matrices in any of $\SOrthogonal{N}$,
  $\SOneg{N}$, $\Unitary{N}$, $\Symplectic{N}$ are a determinantal
  point process, with respect to uniform measure on $\Lambda$, with
  kernels as follows.

  \begin{center}
    \begin{tabular}{|c|c|c|}\hline
      & $K_N(x,y)$ & $\Lambda$ \\
      \hline\hline
      $\SOrthogonal{2N}$ & $\ds 1+\sum_{j=1}^{N-1}2\cos(jx)\cos(jy)$ & $[0,\pi)$ \\
      \hline
      $\SOrthogonal{2N+1}, \SOneg{2N+1}$ 
      & $\ds\sum_{j=0}^{N-1}2\sin\left(\frac{(2j+1)x}{2}\right)\sin\left(\frac{(2j+1)y}{2}\right)$
      & $[0,\pi)$ \\
      \hline
      $\Unitary{N}$ & $\ds\sum_{j=0}^{N-1}e^{ij(x-y)}$ & $[0,2\pi)$ \\
      \hline
      $ \Symplectic{N}, \SOneg{2N+2}$ & $\ds\sum_{j=1}^N2\sin(jx)\sin(jy)$
      & $[0,\pi)$ \\
      \hline
    \end{tabular}
  \end{center}
\end{prop}

\medskip

Proposition \ref{T:evs_detpp} allows us to apply the following result from
\cite{HKPV06}; see also Corollary 4.2.24 of \cite{AGZ}.

\begin{prop} \label{T:detpp_Bernoulli}
  Let $K:\Lambda\times\Lambda\to\C$ be a kernel on a locally compact
  Polish space $\Lambda$ such that the
  corresponding integral operator $\mathcal{K}:L^2(\mu)\to L^2(\mu)$
  defined by
  \[
  \mathcal{K}(f)(x)=\int K(x,y) f(y) \ d\mu(y)
  \]
  is self-adjoint, nonnegative, and locally trace-class with
  eigenvalues in $[0,1]$.  For $D\subseteq\Lambda$ measurable, let $K_D(x,y) =
  \ind{D}(x) K(x,y) \ind{D}(y)$ be the restriction of $K$ to $D$.
  Suppose that
  $D$ is such that $K_D$ is trace-class;
  denote by $\{\lambda_k\}_{k\in \mathcal{A}}$ the eigenvalues of the corresponding operator
  $\mathcal{K}_D$ on $L^2(D)$ ($\mathcal{A}$ may be finite or
  countable) and denote by $\mathcal{N}_D$ the number
  of particles of the determinantal point process with kernel $K$
  which lie in $D$.  Then
  \[
  \mathcal{N}_D \overset{d}{=} \sum_{k\in\mathcal{A}} \xi_k,
  \]
  where ``$\overset{d}{=}$'' denotes equality in distribution and the
  $\xi_k$ are independent Bernoulli random variables with
  $\Prob[\xi_k=1] = \lambda_k$ and $\Prob[\xi_k = 0] = 1 - \lambda_k$.
\end{prop}

\medskip

In order to treat powers of uniform random matrices, we will
make use of the following elegant result of Rains.  For simplicity of
exposition, we will restrict attention for now to the unitary group,
and discuss in Section \ref{S:other_groups} the straightforward
modifications needed to treat the other classical compact groups.

\begin{prop}[Rains, \cite{Ra03}]\label{T:Rains}
  Let $m\le N$ be fixed.  If $\sim$ denotes equality of
  eigenvalue distributions, then
  \[
  \Unitary{N}^m \sim \bigoplus_{0\le
    j<m} \Unitary{\left\lceil\frac{N-j}{m}\right\rceil}
  \] 
  That is, if $U$ is a uniform $N\times N$ unitary matrix, the
  eigenvalues of $U^m$ are distributed as those of $m$ independent
  uniform unitary matrices of sizes $\left\lfloor \frac{N}{m}
  \right\rfloor:=\max\left\{k\in\N\mid k\le\frac{N}{m}\right\}$ and
  $\left\lceil \frac{N}{m} \right\rceil:=\min\left\{k\in\N\mid k\ge\frac{N}{m}\right\}$, such that
  the sum of the sizes of the matrices is $N$.
\end{prop}

\medskip

\begin{cor}\label{T:power_counting_function}
  Let $1 \le m \le N$, and let $U \in \Unitary{n}$ be uniformly
  distributed.  For $\theta \in [0, 2\pi)$, denote by
  $\mathcal{N}^{(m)}_\theta$ the number of eigenvalue angles of $U^m$ which
  lie in $[0, \theta)$.  Then $\mathcal{N}^{(m)}_\theta$ is equal in
  distribution to a sum of independent Bernoulli random variables.
  Consequently, for each $t > 0$,
  \begin{equation}\label{E:Bernstein}
    \Prob\left[\abs{\mathcal{N}^{(m)}_\theta - \E \mathcal{N}^{(m)}_\theta} > t \right]
    \le 2 \exp \left(-\min\left\{\frac{t^2}{4\sigma^2},
        \frac{t}{2}\right\}\right),
  \end{equation}
  where $\sigma^2 = \var \mathcal{N}^{(m)}_\theta$.
\end{cor}

\begin{proof}
  By Proposition \ref{T:Rains}, $\mathcal{N}^{(m)}_\theta$ is equal to
  the sum of $m$ independent random variables $X_i$, $1 \le i \le m$,
  which count the number of eigenvalue angles of smaller-rank
  uniformly distributed unitary matrices which lie in the interval.
  Propositions \ref{T:evs_detpp} and \ref{T:detpp_Bernoulli} together
  imply that each $X_i$ is equal in distribution to a sum of
  independent Bernoulli random variables, which completes the proof of
  the first claim. The inequality \eqref{E:Bernstein} then follows
  immediately from Bernstein's inequality \cite[Lemma
  2.7.1]{Talagrand}.
\end{proof}

\medskip

\section{Means and variances}\label{S:meansvars}

In order to apply \eqref{E:Bernstein}, it is necessary to estimate the
mean and variance of the eigenvalue counting function
$\mathcal{N}^{(m)}_\theta$.  As in the proof of Corollary
\ref{T:power_counting_function}, this reduces by Proposition
\ref{T:Rains} to considering the case $m = 1$.  Asymptotics for these
quantities have been stated in the literature before, e.g.\ in
\cite{So00}, but not with the uniformity in $\theta$ which is needed
below, so we indicate one approach to the proofs.  A different
approach yielding very precise asymptotics was carried out by Rains
\cite{Ra97} for the unitary group; we use the approach outlined below
because it generalizes easily to all of the other groups and cosets.

For this purpose we again make use of the fact that the eigenvalue
distributions of these random matrices are determinantal point
processes.  It is more convenient for the variance estimates to use
here an alternative representation to the one stated in Proposition
\ref{T:evs_detpp} (which is more convenient for verifying the
hypotheses of Proposition \ref{T:detpp_Bernoulli} and for the mean
estimates).  First define
\[
S_N(x):= \begin{cases} \sin\left(\frac{Nx}{2}\right) /
  \sin\left(\frac{x}{2}\right) & \text{ if } x \neq 0, \\
  N & \text{ if } x = 0.\end{cases}
\]
The following result essentially summarizes \cite[Section 5.4]{KaSa}.
(Note that in the unitary case, the kernels given in Propositions
\ref{T:evs_detpp} and \ref{T:evs_detpp2} are not actually equal, but
they generate the same process).

\begin{prop} \label{T:evs_detpp2} The nontrivial eigenvalue angles of
  uniformly distributed random matrices in any of $\SOrthogonal{N}$,
  $\SOneg{N}$, $\Unitary{N}$, $\Symplectic{N}$ are a determinantal
  point process, with respect to uniform measure on $\Lambda$, with
  kernels as follows.

  \begin{center}
    \begin{tabular}{|c|c|c|}\hline
      & $K_N(x,y)$ & $\Lambda$ \\
      \hline\hline
      $\SOrthogonal{2N}$ 
      & $\ds \frac{1}{2}\Bigl(S_{2N-1}(x-y)+S_{2N-1}(x+y)\Bigr)\phantom{\Bigg|}$&$[0,\pi)$ \\
      \hline
      $\SOrthogonal{2N+1},\SOneg{2N+1}$ 
      & $\ds\frac{1}{2}\Bigl(S_{2N}(x-y)-S_{2N}(x+y)\Bigr) \phantom{\Bigg|}$
      & $[0,\pi)$ \\
      \hline
      $ \Unitary{N}$ & $\ds S_N(x-y)\phantom{\Big|}$ & $[0,2\pi)$ \\
      \hline
      $ \Symplectic{N},\SOneg{2N+2}$
      & $\ds\frac{1}{2}\Bigl(S_{2N+1}(x-y)-S_{2N+1}(x+y)\Bigr) \phantom{\Bigg|}$
      & $[0,\pi)$ \\
      \hline
    \end{tabular}
  \end{center} 
\end{prop}

\medskip

The following lemma is easy to check using Proposition
\ref{T:detpp_Bernoulli}.  For the details of the variance expression,
see \cite[Appendix B]{Gu}.

\begin{lemma} \label{T:kernel_mean_var} Let $K:I \times I \to \R$ be a
  continuous kernel on an interval $I$ representing an orthogonal
  projection operator on $L^2(\mu)$, where $\mu$ is the uniform
  measure on $I$. For a subinterval $D \subseteq I$, denote by
  $\mathcal{N}_D$ the number of particles of the determinantal point
  process with kernel $K$ which lie in $D$. Then
  \[
  \E \mathcal{N}_D = \int_D K(x, x) \ d\mu(x)
  \]
  and
  \[
  \var \mathcal{N}_D = \int_D \int_{I \setminus D} K(x, y)^2 \
  d\mu(x) \ d\mu(y).
  \]
\end{lemma}

\medskip

\begin{prop} \label{T:means}
  \begin{enumerate}
  \item Let $U$ be uniform in $\Unitary{N}$.  For $\theta \in
    [0,2\pi)$, let $\mathcal{N}_\theta$ be the number of eigenvalues
    angles of $U$ in $[0, \theta)$. Then
    \[
    \E \mathcal{N}_\theta = \frac{N\theta}{2\pi}.
    \]
  \item Let $U$ be uniform in one of $\SOrthogonal{2N}$,
    $\SOneg{2N+2}$, $\SOrthogonal{2N+1}$, $\SOneg{2N+1}$, or
    $\Symplectic{N}$.  For $\theta \in [0,\pi)$, let
    $\mathcal{N}_\theta$ be the number of nontrivial eigenvalue angles
    of $U$ in $[0, \theta)$.  Then
    \[
    \abs{\E \mathcal{N}_\theta - \frac{N\theta}{\pi}}
    < 1.
    \]
   \end{enumerate}
\end{prop}

\begin{proof}
  The equality for the unitary group follows from symmetry
  considerations, or immediately from Proposition \ref{T:evs_detpp2}
  and Lemma \ref{T:kernel_mean_var}.

  In the case of $\Symplectic{N}$ or $\SOneg{2N+2}$, by Proposition
  \ref{T:evs_detpp} and Lemma \ref{T:kernel_mean_var},
  \[
  \E \mathcal{N}_\theta 
  = \frac{1}{\pi} \int_0^\theta \sum_{j=1}^N 2 \sin^2(jx) \ dx
  = \frac{N\theta}{\pi} - \frac{1}{2\pi} 
  \sum_{j=1}^N \frac{\sin (2j\theta)}{j}.
  \]
  Define $a_0 = 0$ and $a_j = \sum_{k=1}^j \sin(2 k \theta)$.
  Then by summation by parts,
  \[
  \sum_{j=1}^N \frac{\sin(2 j\theta)}{j} 
  = \frac{a_N}{N} + \sum_{j=1}^{N-1} \frac{a_j}{j(j+1)}.
  \]
  Trivially, $\abs{a_N} \le N$. Now observe that
  \[
  a_j = \Im \left[ e^{2i\theta} \sum_{k=0}^{j-1} e^{2ik\theta} \right]
  = \Im \left[ e^{2 i\theta} \frac{e^{2ij \theta} - 1}{e^{2i\theta - 1}} \right] 
  = \Im \left[ e^{i (j+1) \theta} \frac{\sin(j\theta)}{\sin (\theta)} \right]
  = \frac{\sin ((j+1)\theta) \sin(j \theta)}{\sin (\theta)}.
  \]
  Since $a_j$ is invariant under the substitution $\theta \mapsto \pi
  - \theta$, it suffices to assume that $0 < \theta \le \pi/2$.  In that
  case $\sin (\theta) \ge 2\theta/\pi$, and so
  \[
    \sum_{j=1}^{N-1} \frac{\abs{a_j}}{j(j+1)}
    \le \frac{\pi}{2\theta} \left[\sum_{1 \le j \le 1/\theta} \theta^2 
      + \sum_{1/\theta < j \le N-1} \frac{1}{j(j+1)}\right]
    \le \frac{\pi}{2\theta} (\theta + \theta) = \pi.
  \]
  All together,
  \[
  \abs{\E \mathcal{N}_\theta - \frac{N\theta}{\pi}} \le \frac{1 +
    \pi}{2\pi}.
  \]
  The other cases are handled similarly.
\end{proof}

\medskip

As before, we restrict attention from now on to the unitary group, deferring
discussion of the other cases to Section \ref{S:other_groups}.

\begin{prop} \label{T:unitary_var} Let $U$ be uniform in
  $\Unitary{N}$.  For $\theta \in [0,2\pi)$, let $\mathcal{N}_\theta$
  be the number of eigenvalue angles of $U$ in $[0, \theta)$.  Then
  \[
  \var \mathcal{N}_\theta \le \log N + 1.
  \]
\end{prop}

\begin{proof}
  If $\theta\in(\pi,2\pi)$, then
  $\mathcal{N}_\theta\overset{d}{=}N-\mathcal{N}_{2\pi-\theta}$, and
  so it suffices to assume that $\theta\le\pi$.  By Proposition
  \ref{T:evs_detpp2} and Lemma \ref{T:kernel_mean_var},
  \begin{align*}
    \var \mathcal{N}_\theta
    & = \frac{1}{4\pi^2} \int_0^\theta \int_\theta^{2\pi}  S_N(x-y)^2 \ dx \ dy 
    = \frac{1}{4\pi^2}\int_0^\theta \int_{\theta-y}^{2\pi-y}
    \frac{\sin^2\left(\frac{Nz}{2}\right)}{\sin^2\left(\frac{z}{2}\right)}
    \ dz\ dy\\
    & = \frac{1}{4\pi^2} \left[\int_0^\theta 
      \frac{z\sin^2 \left(\frac{Nz}{2}\right)}
      {\sin^2 \left(\frac{z}{2}\right)} \ dz
      + \int_\theta^{2\pi-\theta} 
      \frac{\theta\sin^2 \left(\frac{Nz}{2}\right)}
      {\sin^2\left(\frac{z}{2}\right)}\ dz 
      +\int_{2\pi-\theta}^{2\pi} 
      \frac{(2\pi-z)\sin^2\left(\frac{Nz}{2}\right)}
      {\sin^2\left(\frac{z}{2}\right)}\ dz\right] \\
    & = \frac{1}{2\pi^2} \left[ \int_0^\theta
      \frac{z\sin^2\left(\frac{Nz}{2}\right)}
      {\sin^2\left(\frac{z}{2}\right)} \ dz
      + \int_\theta^{\pi} \frac{\theta\sin^2\left(\frac{Nz}{2}\right)}
      {\sin^2\left(\frac{z}{2}\right)} \ dz \right].
  \end{align*}
  For the first integral, since
  $\sin\left(\frac{z}{2}\right) \ge \frac{z}{\pi}$ for all
  $z\in[0,\theta]$, if $\theta>\frac{1}{N}$, then
  \[
  \int_0^\theta \frac{z\sin^2\left(\frac{Nz}{2}\right)}
  {\sin^2\left(\frac{z}{2}\right)} \ dz
  \le \int_0^{\frac{1}{N}} \frac{(\pi N)^2z}{4} \ dz
  + \int_{\frac{1}{N}}^\theta \frac{\pi^2}{z} \ dz
  =\pi^2\left(\frac{1}{8}+\log(N)+\log(\theta)\right).
  \]
  If $\theta\le\frac{1}{N}$, there is no need to break up the integral
  and one simply has the bound $\frac{(\pi
    N\theta)^2}{8}\le\frac{\pi^2}{8}$.  Similarly, if
  $\theta<\frac{1}{N}$, then
  \begin{align*}
    \int_\theta^{\pi} 
    \frac{\theta\sin^2\left(\frac{Nz}{2}\right)}{\sin^2\left(\frac{z}{2}\right)}
    \ dz 
    & \le \int_\theta^{\frac{1}{N}}\frac{\theta  (\pi
      N)^2}{4} \ dz
    +\int_{\frac{1}{N}}^\pi \frac{\pi^2\theta}{z^2}\ dz\\
    & = \frac{\pi^2\theta N}{4}\left(1-N\theta\right)
    + \pi^2N\theta - \pi\theta \le \frac{5\pi^2}{4};
  \end{align*}
  if $\theta\ge\frac{1}{N}$, there is no need to break up the integral
  and one simply has a bound of $\pi^2$.
  
  All together,
  \[
  \var \mathcal{N}_\theta \le \log(N) + \frac{11}{16}.
  \qedhere
  \]
\end{proof}

\medskip

\begin{cor} \label{T:power_mean_var} Let $U$ be uniform in
  $\Unitary{N}$ and $1 \le m \le N$.  For $\theta \in [0,2\pi)$, let
  $\mathcal{N}^{(m)}_\theta$ be the number of eigenvalue angles of $U^m$ in
  $[0, \theta)$.  Then
  \[
  \E \mathcal{N}^{(m)}_\theta = \frac{N \theta}{2\pi}
  \qquad \text{and} \qquad
  \var \mathcal{N}^{(m)}_\theta \le m \left(\log \left(\frac{N}{m}\right) + 1\right).
  \]
\end{cor}

\begin{proof}
By Proposition \ref{T:Rains}, $\mathcal{N}_\theta^{(m)}$ is equal in
distribution to the total number of eigenvalue angles in $[0,\theta)$ of each of
$U_0\ldots,U_{m-1}$, where $U_0,\ldots,U_{m-1}$ are independent and
$U_j$ is uniform in $\Unitary{\left\lceil\frac{N-j}{m}\right\rceil}$; that is,
\[\mathcal{N}_\theta^{(m)}\overset{d}=\sum_{j=0}^{m-1}\mathcal{N}_{j,\theta},\]
where the $\mathcal{N}_{j,\theta}$ are the independent counting
functions corresponding to $U_0,\ldots,U_{m-1}$.

The bounds in the corollary are thus automatic from Propositions \ref{T:means} and
\ref{T:unitary_var}. (Note that the $N/m$ in the
variance bound, as opposed to the more obvious $\lceil N/m \rceil$,
follows from the concavity of the logarithm.)
\end{proof}

\section{Wasserstein distances}\label{S:wasserstein}

In this section we prove bounds and concentration inequalities for the
spectral measures of fixed powers of uniform random unitary
matrices.  The method generalizes the approach taken in \cite{Da12} to
bound the distance of the spectral measure of the Gaussian unitary
ensemble from the semicircle law.

Recall that for $p \ge 1$, the $L_p$-Wasserstein distance between two
probability measures $\mu$ and $\nu$ on $\C$ is defined by
\[
W_p(\mu, \nu) = \left(\inf_{\pi \in \Pi(\mu, \nu)} \int \abs{w - z}^p
  \ d\pi(w, z) \right)^{1/p},
\]
where $\Pi(\mu, \nu)$ is the set of all probability measures on $\C
\times \C$ with marginals $\mu$ and $\nu$.

\begin{lemma} \label{T:evs_concentration} Let $1 \le m \le N$ and let
  $U \in \Unitary{N}$ be uniformly distributed. Denote by $e^{i \theta_j}$,
  $1 \le j \le N$, the eigenvalues of $U^m$, ordered so that $0 \le
  \theta_1 \le \dotsb \le \theta_N < 2\pi$.  Then for each $j$ and $u > 0$,
  \begin{equation}\label{E:individual_concentration}
  \Prob\left[\abs{\theta_j - \frac{2\pi j}{N}} > \frac{4\pi}{N} u\right] \le
  4 \exp\left[- \min \left\{\frac{u^2}{m\left(\log
        \left(\frac{N}{m}\right) + 1\right)}, u \right\}\right].
  \end{equation}
\end{lemma}

\begin{proof}
  For each $1\le j\le N$ and $u > 0$, if $j + 2u < N$ then
  \begin{align*}
    \Prob\left[\theta_j > \frac{2\pi j}{N} + \frac{4\pi}{N} u\right]
    & = \Prob\left[\mathcal{N}^{(m)}_{\frac{2\pi (j+2u)}{N}}<j\right]
    = \Prob\left[j + 2u - \mathcal{N}^{(m)}_{\frac{2\pi (j+2u)}{N}} > 
      2u \right] \\
    & \le \Prob\left[ \abs{\mathcal{N}^{(m)}_{\frac{2\pi (j+2u)}{N}} 
        - \E \mathcal{N}^{(m)}_{\frac{2\pi (j+2u)}{N}}} > 2u \right].
  \end{align*}
  If $j+2u \ge N$ then
  \[
  \Prob\left[\theta_j > \frac{2\pi j}{N} + \frac{4\pi}{N} u\right] 
  = \Prob \left[\theta_j > 2\pi \right]
  = 0,
  \]
  and the above inequality holds trivially.  The probability that
  $\theta_j < \frac{2\pi j}{N} - \frac{4\pi}{N}u$ is bounded in the
  same way.  Inequality \eqref{E:individual_concentration} now follows
  from Corollaries \ref{T:power_counting_function} and \ref{T:power_mean_var}.
\end{proof}

\medskip

\begin{thm}\label{T:mean-dist}
  Let $\mu_{N,m}$ be the spectral measure of $U^m$, where $1 \le m \le N$
  and $U \in \Unitary{N}$ is uniformly distributed, and let $\nu$ denote
  the uniform measure on $\Circle$. Then for each $p \ge 1$,
  \[
  \E W_p (\mu_{N,m},\nu) 
  \le C p \frac{\sqrt{m\left[\log \left(\frac{N}{m}\right) + 1 \right]}}{N},
  \]
  where $C > 0$ is an absolute constant.
\end{thm}

\begin{proof}
  Let $\theta_j$ be as in Lemma \ref{T:evs_concentration}. Then by
  Fubini's theorem,
  \begin{align*}
    \E \abs{\theta_j-\frac{2\pi j}{N}}^p 
    & = \int_0^\infty p t^{p-1} \Prob\left[\abs{\theta_j - \frac{2\pi j}{N}} > t 
    \right] \ dt \\
    & =\frac{(4\pi)^p p}{N^p} \int_0^\infty u^{p-1} 
    \Prob\left[\abs{ \theta_j-\frac{2\pi j}{n}}
      > \frac{4\pi}{N} u \right] \ du \\
    & \le \frac{4(4\pi)^p p}{N^p} \left[
      \int_0^\infty u^{p-1} e^{-u^2/ m[\log (N/m) + 1]} \ du
    + \int_0^\infty u^{p-1} e^{-u} \ du \right]\\
    & = \frac{4(4\pi)^p}{N^p} \left[
      \left(m\left[\log\left(\frac{N}{m}\right) + 1 \right]\right)^{p/2} 
      \Gamma\left(\frac{p}{2} + 1\right) + \Gamma(p+1) \right]\\
    & \le 8 \Gamma(p+1) \left(\frac{4\pi}{N} 
      \sqrt{m\left[\log\left(\frac{N}{m} \right) + 1 \right]}\right)^p.
  \end{align*}
  Observe that in particular,
  \[
  \var \theta_j \le C \frac{ m \left[\log \left(\frac{N}{m}\right) + 1 \right]}
  {N^2}.
  \]

  Let $\nu_N$ be the measure which puts mass $\frac{1}{N}$ at each of
  the points $e^{2\pi i j/N}$, $1 \le j \le N$. Then
  \begin{align*}
      \E W_p(\mu_{N,m}, \nu_N)^p
      & \le \E \left[\frac{1}{N} \sum_{j=1}^N \abs{e^{i \theta_j} - e^{2\pi i j/N}}^p
        \right]
      \le \E \left[\frac{1}{N}\sum_{j=1}^N \abs{\theta_j-\frac{2\pi j}{N}}^p
      \right] \\
      & \le 8 \Gamma(p+1) \left(\frac{4\pi}{N} 
      \sqrt{m\left[\log\left(\frac{N}{m} \right) + 1 \right]}\right)^p.
    \end{align*}
  It is easy to check that $W_p(\nu_N,\nu) \le
  \frac{\pi}{N}$, and thus
\[\E W_p(\mu_{N,m},\nu)\le \E
W_p(\mu_{N,m},\nu_N)+\frac{\pi}{N}\le\left(\E W_p(\mu_{N,m},\nu_N)^p\right)^{\frac{1}{p}}+\frac{\pi}{N}.\]
Applying Stirling's formula to bound $\Gamma(p+1)^{\frac{1}{p}}$ completes the
  proof.
\end{proof}

\medskip

In the case that $m=1$ and $p\le 2$, Theorem \ref{T:mean-dist} could
now be combined with Corollary 2.4 and Lemma 2.5 from \cite{MM12} in
order to obtain a sharp concentration inequality for
$W_p(\mu_{N,1},\nu)$.  However, for $m > 1$ we did not prove an
analogous concentration inequality for $W_p(\mu_{N,m},\nu)$ because
the main tool needed to carry out the approach taken in \cite{MM12},
specifically, a logarithmic Sobolev inequality on the full unitary
group, was not available.  The appendix to this paper contains the
proof of the necessary logarithmic Sobolev inequality on the unitary
group (Theorem \ref{T:LSI}) and the approach to concentration taken in
\cite{MM12}, in combination with Proposition \ref{T:Rains}, can then
be carried out in the present context.

The following lemma, which generalizes part of \cite[Lemma 2.3]{MM12},
provides the necessary Lipschitz estimates for the functions to which
the concentration property will be applied.

\begin{lemma}\label{T:Lipschitz-estimates}
  Let $p \ge 1$. The map $A \mapsto \mu_A$ taking an $N\times N$
  normal matrix to its spectral measure is Lipschitz with constant
  $N^{-1/\max\{p, 2\}}$ with respect to $W_p$. Thus if $\rho$ is any
  fixed probability measure on $\C$, the map $A \mapsto W_p(\mu_A,
  \rho)$ is Lipschitz with constant $N^{-1/\max\{p, 2\}}$.
\end{lemma}

\begin{proof}
  If $A$ and $B$ are $N \times N$ normal matrices, then the
  Hoffman--Wielandt inequality \cite[Theorem VI.4.1]{Bh} states
  that
  \begin{equation}\label{E:permutations_HS_bound}
    \min_{\sigma \in \Sigma_N} \sum_{j=1}^N \abs{\lambda_j(A) -
      \lambda_{\sigma(j)}(B)}^2 \le \norm{A - B}_{HS}^2, 
  \end{equation}
  where $\lambda_1(A), \dotsc, \lambda_N(A)$ and $\lambda_1(B),
  \dotsc, \lambda_N(B)$ are the eigenvalues (with multiplicity, in any
  order) of $A$ and $B$ respectively, and $\Sigma_N$ is the group of
  permutations on $N$ letters. Defining couplings of $\mu_A$ and
  $\mu_B$ given by
  \[
  \pi_\sigma = \frac{1}{N} \sum_{j=1}^N \delta_{(\lambda_j(A),
    \lambda_{\sigma(j)}(B))}
  \]
  for $\sigma \in \Sigma_N$, it follows from
  \eqref{E:permutations_HS_bound} that
  \begin{align*}
    W_p(\mu_A, \mu_B)
    & \le \min_{\sigma \in \Sigma_N} \left( \frac{1}{N}
    \sum_{j=1}^N \abs{\lambda_j(A) - \lambda_{\sigma(j)}(B)}^p \right)^{1/p} \\
    & \le N^{-1/\max\{p, 2\}}\min_{\sigma \in \Sigma_N} \left(
    \sum_{j=1}^N \abs{\lambda_j(A) - \lambda_{\sigma(j)}(B)}^2 \right)^{1/2} \\
    &\le  N^{-1/\max\{p, 2\}} \norm{A - B}_{HS}.
    \qedhere
  \end{align*}
\end{proof}

\medskip

\begin{thm} \label{T:dist-concentration}
  Let $\mu_{N,m}$ be the empirical spectral measure of $U^m$, where $U
  \in \Unitary{N}$ is uniformly distributed and $1 \le m\le N$, and let
  $\nu$ denote the uniform probability measure on $\Circle$.  Then for
  each $t > 0$,
  \[
  \Prob \left[W_p(\mu_{N,m},\nu) \ge C
    \frac{\sqrt{m\left[\log\left(\frac{N}{m}\right) + 1 \right]}}{N} 
    + t \right] 
  \le \exp\left[-\frac{N^2 t^2}{24 m} \right]
  \]
  for $1 \le p \le 2$ and
  \[
  \Prob \left[W_p(\mu_{N,m},\nu) 
    \ge C p \frac{\sqrt{m \left[\log\left(\frac{N}{m}\right) + 1 \right]}}{N} 
  + t \right]
  \le \exp\left[-\frac{N^{1+2/p}t^2}{24 m} \right]
  \]
  for $p > 2$, where $C > 0$ is an absolute constant.
\end{thm}

\begin{proof}
  By Proposition \ref{T:Rains}, $\mu_{N,m}$ is equal in distribution
  to the spectral measure  of a block-diagonal $N \times N$ random matrix $U_1 \oplus
  \dotsb \oplus U_m$, where the $U_j$ are independent and uniform in
  $\Unitary{\left\lfloor \frac{N}{m} \right\rfloor}$ and
  $\Unitary{\left\lceil \frac{N}{m} \right\rceil}$.  Identify
  $\mu_{N,m}$ with this measure and define the function $F(U_1,
  \ldots, U_m) = W_p(\mu_{U_1\oplus\dotsb\oplus U_m},\nu)$; the
  preceding discussion means that if $U_1,\ldots,U_m$ are independent
  and uniform in $\Unitary{\left\lfloor \frac{N}{m} \right\rfloor}$ and
  $\Unitary{\left\lceil \frac{N}{m} \right\rceil}$ as necessary, then
  $F(U_1,\ldots,U_m)\overset{d}=W_p(\mu_{N,m}, \nu)$.  

Applying the concentration inequality in Corollary
  \ref{T:unitary-concentration} of the appendix to the function $F$ gives that
\[ \Prob \bigl[ F(U_1, \dotsc, U_m) \ge \E F(U_1, \dotsc, U_m) + t \bigr]
  \le e^{- N t^2 / 24m L^2},\] 
where $L$ is the Lipschitz constant of $F$, and we have used the
trivial estimate $\left\lfloor\frac{N}{m}\right\rfloor\ge\frac{N}{2m}$.
Inserting the estimate of $\E F(U_1,\ldots,U_m)$ from
  Theorem \ref{T:mean-dist} and the Lipschitz estimates of Lemma
  \ref{T:Lipschitz-estimates} completes the proof.
\end{proof}

\medskip

\begin{cor}\label{T:unitary-BC}
  Suppose that for each $N$, $U_N \in \Unitary{N}$ is uniformly distributed
  and $1 \le m_N \le N$. Let $\nu$ denote the uniform measure on
  $\Circle$. There is an absolute constant $C$ such that given $p \ge
  1$, with probability $1$, for all sufficiently large $N$,
  \[
  W_p(\mu_{N, m_N}, \nu) \le C \frac{\sqrt{m_N \log(N)}}{N}
  \]
  if $1 \le p \le 2$ and
  \[
  W_p(\mu_{N, m_N}, \nu) \le C p \frac{\sqrt{m_N \log(N)}}
  {N^{\frac{1}{2} + \frac{1}{p}}}
  \]
  if $p > 2$.
\end{cor}

\begin{proof}
  In Theorem \ref{T:dist-concentration} let $t_N = 5 \frac{\sqrt{m_N
      \log(N)}}{N}$ for $p\in[1,2]$ and $t_N = 5\frac{\sqrt{m_N
      \log(N)}}{N^{\frac{1}{2} + \frac{1}{p}}}$ for $p>2$, and apply
  the Borel--Cantelli lemma.
\end{proof}

We observe that Corollary \ref{T:unitary-BC} makes no assumption
about any joint distribution of the matrices $\{U_N\}_{N\in\N}$; in
particular, they need not be independent.

\bigskip

As a final note, Rains's Proposition \ref{T:Rains} above shows that, in the
case $m=N$, $\mu_{N,m}$ is the empirical measure of $N$ i.i.d.\ points
on $\Circle$.  By another result of Rains \cite{Ra97}, the same is
true when $m > N$.  In particular, in all the above results the
restriction $m \le N$ may be removed if $m$ is simply replaced by
$\min \{m,N\}$ in the conclusion.

\section{Other groups}\label{S:other_groups}

The approach taken above can be completed in essentially the same way
for $\SOrthogonal{N}$, $\SOneg{N}$ and $\Symplectic{N}$, so that all
the results above hold in those cases as well, with only the precise
values of constants changed.

In \cite{Ra03}, Rains proved that the eigenvalue distributions for
these groups (or rather, components, in the case of $\SOneg{N}$) can
be decomposed similarly to the decomposition described in Proposition
\ref{T:Rains}, although the decompositions are more complicated in
those cases (mostly because of parity issues).  The crucial fact,
though, is that the decomposition is still in terms of independent
copies of smaller-rank (orthogonal) groups and cosets.  This allows
for the representation of the eigenvalue counting function in all
cases as a sum of independent Bernoulli random variables (allowing for
the application of Bernstein's inequality) and as a sum of independent
copies of eigenvalue counting functions for smaller-rank groups.  In
particular the analogue of Corollary \ref{T:power_counting_function}
holds and it suffices to estimate the means and variances in the case
$m=1$.  The analogue of Proposition \ref{T:unitary_var} for the other
groups can be proved similarly using Proposition \ref{T:evs_detpp2}
and Lemma \ref{T:kernel_mean_var}.

With those tools and Proposition \ref{T:means} on hand, the analogue
of Theorem \ref{T:mean-dist} can be proved in the same way, with a
minor twist. One can bound as in the proof of Theorem
\ref{T:mean-dist} the distance between the empirical measure
associated to the nontrivial eigenvalues and the uniform measure on
the upper-half circle. Since the nontrivial eigenvalues occur in
complex conjugate pairs and there are at most two trivial eigenvalues,
one gets essentially the same bound for the distance between the
empirical spectral measure and the uniform measure on the whole
circle.

Finally, logarithmic Sobolev inequalities --- and hence concentration
results analogous to Corollary \ref{T:unitary-concentration} --- for
the other groups are already known via the Bakry--\'Emery criterion,
cf.\ \cite[Section 4.4]{AGZ}, so that the analogue of Theorem
\ref{T:dist-concentration} follows as for the unitary group.

For the special unitary group $\SUnitary{N}$, all the results stated
above hold exactly as stated for the full unitary group, cf.\ the
proof of \cite[Lemma 2.5]{MM12}.  Analogous results for the full
orthogonal group $\Orthogonal{N}$ follow from the results for
$\SOrthogonal{N}$ and $\SOneg{N}$ by conditioning on the determinant,
cf.\ the proofs of Theorem 2.6 and Corollary 2.7 in \cite{MM12}.

\section*{Appendix: the log-Sobolev constant of
the unitary group}

In this section we prove a logarithmic Sobolev inequality for the
unitary group with a constant of optimal order.  As a consequence, we
obtain a sharp concentration inequality, independent of $k$, for
functions of $k$ independent unitary random matrices.

Recall the following general definitions for a metric space $(X,d)$
equipped with a Borel probability measure $\mu$.  The entropy of a
measurable function $f:X \to [0, \infty)$ with respect to
$\mu$ is
\[
\Ent_\mu(f) := \int f\log(f) \ d\mu - \left(\int f \ d\mu \right)
\log \left(\int f \ d\mu\right).
\]
For a locally Lipschitz function $g:X \to \R$,
\[
\abs{\nabla g}(x) := \limsup_{y\to x}\frac{\abs{g(y)-g(x)}}{d(y,x)}.
\]
We say that $(X,d,\mu)$ satisfies a logarithmic Sobolev inequality (or
log-Sobolev inequality for short) with constant $C > 0$ if, for every
locally Lipschitz $f:X\to\R$,
\begin{equation}\label{E:LSI}
\Ent_\mu (f^2) \le 2 C \int \abs{\nabla f}^2
\ d\mu.
\end{equation}

\begin{thm} \label{T:LSI} The unitary group $\Unitary{N}$, equipped
  with its uniform probability measure and the Hilbert--Schmidt
  metric, satisfies a logarithmic Sobolev inequality with constant
  $6/N$.
\end{thm}

\medskip

If the Riemannian structure on $\Unitary{N}$ is %given the normalization 
the one induced by the usual Hilbert--Schmidt inner product on
matrix space, the the geodesic distance is bounded above by $\pi/2$
times the Hilbert--Schmidt distance on $\Unitary{N}$ (see e.g.\
\cite[Lemma 3.9.1]{Blower}). Thus Theorem \ref{T:LSI} implies that
$\Unitary{N}$ equipped with the geodesic distance also satisfies a
log-Sobolev inequality, with constant $3 \pi^2/ 2N$.

It is already known that every compact Riemannian manifold, equipped
with the normalized Riemannian volume measure and geodesic distance,
satisfies a log-Sobolev inequality with some finite constant
\cite{Rot2}.  For applications like those in this paper to a sequence
of manifolds such as $\{\Unitary{N}\}_{N=1}^\infty$, however, the
order of the constant as $N$ grows is crucial.  The constant in
Theorem \ref{T:LSI} is best possible up to a constant factor; this can
be seen, for example, from the fact that one can recover the sharp
concentration of measure phenomenon on the sphere from
Corollary \ref{T:unitary-concentration} below.

The key to the proof of Theorem \ref{T:LSI} is the following
representation of uniform measure on the unitary group.

\begin{lemma} \label{T:coupling} Let $\theta$ be uniformly distributed
  in $\left[0, \frac{2\pi}{N}\right]$ and let $V \in \SUnitary{N}$ be
  uniformly distributed, with $\theta$ and $V$ independent.  Then
  $e^{i\theta}V$ is uniformly distributed in $\Unitary{N}$.
\end{lemma}

\begin{proof}
  Let $X$ be uniformly distributed in $[0,1)$, $K$ uniformly
  distributed in $\{0, \ldots, N-1\}$, and $V$ uniformly distributed
  in $\SUnitary{N}$ 
  with $(X,K,V)$ independent.  Consider
  \[
  U = e^{2\pi iX/N}e^{2\pi i K/N}V.
  \] 
  
  On one hand, it is easy to see that $(X+K)$ is uniformly distributed
  in $[0,N]$, so that $e^{2\pi i (X+K)/N}$ is uniformly distributed on
  $\Circle$. Thus $U\overset{d}=\omega V$ for $\omega$ uniform in
  $\Circle$ and independent of $V$. It is then straightforward to see
that $U \in \Unitary{N}$ is uniformly distributed (cf.\ the
  proof of \cite[Lemma 2.5]{MM12}).
  
  On the other hand, if $I_N$ is the $N\times N$ identity matrix, then
  $e^{2\pi i K/N} I_N \in \SUnitary{N}$. By the translation invariance
  of uniform measure on $\SUnitary{N}$ this implies that $e^{2\pi i K
    / N}V \overset{d}=V$, and so $e^{2\pi i X/N} V \overset{d}=U$.
\end{proof}

\medskip

\begin{proof}[Proof of Theorem \ref{T:LSI}]
  First, for the interval $[0, 2\pi]$ equipped with its standard
  metric and uniform measure, the optimal constant in \eqref{E:LSI}
  for functions $f$ with $f(0) = f(2\pi)$ is known to be $1$, see
  e.g.\ \cite{Wei}. This fact completes the proof --- with a better
  constant than stated above --- in the case $N = 1$, since
  $\Unitary{1} = \Circle$; we assume from now on that $N \ge 2$.  By
  reflection, the optimal constant for general locally Lipschitz
  functions on $[0, \pi]$ is also $1$.  It follows by a scaling
  argument that the optimal logarithmic Sobolev constant on $\left[0,
    \frac{\pi \sqrt{2}}{ \sqrt{N}}\right)$ is $2/N$.

  By the Bakry--\'Emery Ricci curvature criterion \cite{BaEm},
  $\SUnitary{N}$ satisfies a log-Sobolev inequality with constant
  $2/N$ when equipped with its geodesic distance, and hence also when
  equipped with the Hilbert--Schmidt metric (see Section 4.4 and
  Appendix F of \cite{AGZ}).  By the tensorization property of
  log-Sobolev inequalities in Euclidean spaces (see \cite[Corollary
  5.7]{Led}), the product space $\left[0,
    \frac{\pi\sqrt{2}}{\sqrt{N}}\right) \times \SUnitary{N}$, equipped
  with the $L_2$-sum metric, satisfies a log-Sobolev inequality with
  constant $2/N$ as well.

  Define the map $F:\left[0,\frac{\pi\sqrt{2}}{\sqrt{n}}\right) \times
  \SUnitary{N} \to \Unitary{N}$ by $F(t,V) = e^{\sqrt{2} it /
    \sqrt{N}} V$. By Lemma \ref{T:coupling}, the push-forward via $F$
  of the product of uniform measure on $\left[0,
    \frac{\pi\sqrt{2}}{\sqrt{N}}\right) $ with uniform measure on
  $\SUnitary{N}$ is uniform measure on $\Unitary{N}$.  Moreover, this
  map is $\sqrt{3}$-Lipschitz:
  \begin{align*}
    \norm{e^{\sqrt{2} i t_1/\sqrt{N}} V_1 -
      e^{\sqrt{2} i t_2/\sqrt{N}} V_2}_{HS}
    & \le \norm{ e^{\sqrt{2} i t_1/\sqrt{N}} V_1 - 
      e^{\sqrt{2} i t_1/\sqrt{N}} V_2}_{HS} \\
    & \qquad + \norm{e^{\sqrt{2} i t_1/\sqrt{N}} V_2 -
      e^{\sqrt{2} i t_2/\sqrt{N}} V_2}_{HS} \\
    & = \norm{V_1 - V_2}_{HS}
    + \norm{e^{\sqrt{2} i t_1 / \sqrt{N}}I_N -
      e^{\sqrt{2} i t_2 / \sqrt{N}} I_N }_{HS} \\
    & \le \norm{ V_1 - V_2}_{HS}
    + \sqrt{2} \abs{t_1- t_2} \\
    & \le \sqrt{3} \sqrt{\norm{ V_1 - V_2}_{HS}^2
    + \abs{t_1- t_2}^2}.
  \end{align*}

  Since the map $F$ is $\sqrt{3}$-Lipschitz, its image $\Unitary{N}$
  with the (uniform) image measure satisfies a logarithmic Sobolev
  inequality with constant $(\sqrt{3})^2 \frac{2}{N} = \frac{6}{N}$.
\end{proof}

\medskip

\begin{cor} \label{T:unitary-concentration} Given $N_1, \dotsc, N_k
  \in \N$, denote by $M = \Unitary{N_1} \times \dotsb \Unitary{N_k}$
  equipped with the $L_2$-sum of Hilbert--Schmidt metrics.  Suppose
  that $F:M \to \R$ is $L$-Lipschitz, and that $\{ U_j \in
  \Unitary{N_j} : 1 \le j \le k \}$ are independent, uniform random
  unitary matrices.  Then for each $t > 0$,
  \[
  \Prob \bigl[ F(U_1, \dotsc, U_k) \ge \E F(U_1, \dotsc, U_k) + t \bigr]
  \le e^{- N t^2 / 12 L^2},
  \]
  where $N = \min\{N_1, \dotsc, N_k\}$.
\end{cor}

\begin{proof}
  By Theorem \ref{T:LSI} and the tensorization property of log-Sobolev
  inequalities \cite[Corollary 5.7]{Led}, $M$ satisfies a log-Sobolev
  inequality with constant $6/N$.  The stated concentration inequality
  then follows from the Herbst argument (see, e.g., \cite{Led},
  Theorem 5.3).
\end{proof}

\medskip

The lack of dependence on $k$ is a crucial feature of the inequality in
Corollary \ref{T:unitary-concentration}; unlike
logarithmic Sobolev inequalities, concentration inequalities
themselves do not tensorize without introducing a dependence on $k$.

\bibliographystyle{plain}
\bibliography{powers}

\end{document}